\RequirePackage[british]{babel}
\documentclass[a4paper,leqno,twoside]{amsart}
\usepackage{amssymb}
\usepackage{amscd}
\usepackage{url}
\usepackage{tabularx}

\theoremstyle{plain}
\newtheorem{thm}{Theorem}

\newtheorem{prop}[thm]{Proposition}
\newtheorem{lem}[thm]{Lemma}

\newtheorem{pb}{Problem}
\newtheorem*{ta}{Theorem A}
\newtheorem*{tb}{Theorem B}

\theoremstyle{definition}
\newtheorem{df}{Definition}

\theoremstyle{remark}
\newtheorem{rmk}{Remark}
\newtheorem*{acks}{Acknowledgments}

\newcommand{\ie}{\textit{i.e. }}

\newcommand{\p}{\mathbb{P}}

\newcommand{\R}{\mathcal{R}}

\newcommand{\G}{\mathrm{Gr}}

\newcommand{\OO}{\mathcal{O}}

\newcommand{\II}{\mathcal{I}}

\newcommand{\Sym}{\mathcal{S}ym}
\newcommand{\abs}[1]{\lvert#1\rvert}

\newcommand{\gen}[1]{\langle#1\rangle}

\newcommand{\CC}{{\mathbb C}}
\newcommand{\de}{\partial}

\DeclareMathOperator{\Gl}{GL}
\DeclareMathOperator{\bir}{Bir}

\DeclareMathOperator{\Hilb}{Hilb}

\DeclareMathOperator{\Ext}{Ext}

\def\cocoa{{\hbox{\rm C\kern-.13em o\kern-.07em C\kern-.13em o\kern-.15em A}}}
\DeclareMathOperator{\PGL}{\mathbb{P}GL}

\DeclareMathOperator{\proj}{Proj}
\DeclareMathOperator{\spec}{Spec}
\DeclareMathOperator{\VPS}{VPS}

\begin{document}

\title{Gonality, Apolarity and hypercubics}
\author{Pietro De Poi \and Francesco Zucconi}

\thanks{This research was partially supported by MiUR,
project ``Geometria delle variet\`a algebriche e dei loro spazi di moduli'' 
and Regione Friuli Venezia Giulia,
``Progetto D4'' for the first author, and by MiUR,
project ``Spazi di moduli e teorie di Lie'' for the second one}

\address{
Dipartimento di Matematica e Informatica\\
Universit\`a degli Stud\^\i\ di Udine\\
Via delle Scienze, 206\\
Loc. Rizzi\\
33100 Udine\\
Italy
}

\email{pietro.depoi@dimi.uniud.it} \email{francesco.zucconi@dimi.uniud.it}
\date{\today}

\begin{abstract} 
We show that any Fermat hypercubic is apolar to a trigonal curve, and vice versa. 
We show also that the Waring number of the 
polar hypercubic associated to a 
tetragonal curve of genus $g$
is at most $\lceil\frac{3}{2}g-\frac{7}{2}\rceil$, and for an important class of them is 
at most $\frac{4}{3}g-\frac{5}{3}$. 
\end{abstract}
\maketitle

\bibliographystyle{amsalpha}

\section{Introduction} 
Let $C$ be a non-hyperelliptic, smooth, projective curve 
of genus $g$ defined over
$\mathbb C$ and let
$\R_{C}:=\oplus_{i=0}^{\infty}H^{0}(C,\omega_{C}^{\otimes i})$ be its 
canonical ring.
It is well-known that $C$ is isomorphic to the canonical curve $\proj(\R_{C})$, 
which embeds in 
$ \mathbb P^{g-1}$ as a projectively normal variety
and the homogeneous ideal ${\II}_{C}$ of $C$ in $\mathbb P^{g-1}$ 
is generated in degree
$2$ unless $C$ is trigonal or a smooth plane quintic, see 
\cite{ACGH}.

Since Green's
seminal papers, \cite{g1}, \cite{g3} and \cite{g2}, the syzygies of $\R_{C}$ have been studied
deeply by several authors; here we can quote, for instance, \cite{Sch}, \cite{mv}, \cite{v}.

In this paper we follow an approach we learned from \cite{IR}. 
We put $ \mathbb P^{g-1}:= \proj(\mathbb C[\partial_{0},\dotsc,
\partial_{g-1}])$ where $\mathbb C[\partial_{0},\dotsc,
\partial_{g-1}]$ is the polynomial ring generated by the natural
derivations over $\mathbb C[x_{0},\dotsc,
x_{g-1}]$. We consider $\eta_{1},\eta_{2}$ two general
linear forms in $\mathbb C[\partial_{0},\dotsc,
\partial_{g-1}]$ which can be assumed to be $\eta_{1}=\partial_{g-1}$ and
$\eta_{2}=\partial_{g-2}$ and we construct the ring $A:=
\frac{\R_{C}}{\gen{\eta_{1},\eta_{2}}}$. An 
easy computation shows that $A$ is an Artinian graded Gorenstein ring of
socle degree $3$. Then, by a result of Macaulay, it can be
realised as $A=\frac{\mathbb C[\partial_{0},\dotsc,
\partial_{g-3}]}{F_{\eta_{1},\eta_{2}}^\perp}$
where $F_{\eta_{1},\eta_{2}}\in \mathbb C[x_{0},\dotsc, x_{g-3}]$ 
is a cubic homogeneous
polynomial and $F_{\eta_{1},\eta_{2}}^\perp:=
\{ D\in \mathbb C[\partial_{0},\dotsc, 
\partial_{g-3}]\vert D(F_{\eta_{1},\eta_{2}})=0 \}$.

In this way, it remains defined a rational map: 
\begin{equation}\label{eq:ac} 
\alpha_C\colon\G(g-2,g)\dashrightarrow H_{g-3,3}
\end{equation}
where $\G(g-2,g)$ is the Grassmannian of $(g-3)$-planes in $\p^{g-1}$ and $H_{g-3,3}$ is the space of hypercubics 
in $\check\p^{g-3}$ modulo the action of $\PGL(g-2,\CC)$.

We will analyse this map when $C$ has a $g^1_n$, that is, it is an $n:1$ covering of 
the projective line. In particular, 
the above construction applied to the case of trigonal curves, \ie $n=3$, gives a
nice correspondence to the Fermat cubics in $g-2$ variables, and vice versa:

\begin{ta}\label{Theorem A}
A canonical curve $C$ is either trigonal or isomorphic 
to a smooth plane quintic 
if and only if $F_{\eta_{1},\eta_{2}}\in \mathbb C[x_{0},\dotsc, x_{g-3}]$ 
is a Fermat cubic, where 
$\eta_1,\eta_2\in H^{0}(C,\omega_{C})$ are general $1$-forms. 
\end{ta}

Theorem A confirms the intuition that the more special is
the curve, the more special is the cubic; in other words, that the
image of the map $\alpha_C$ depends on the
geometry of $C$.

In general, our method to estimate the \emph{Waring number} of $F_{\eta_1,\eta_2}$ requires 
to find the degree of 
a surface $S$ such that $C\subset S\subset X$, where $X$ is a rational normal 
scroll of dimension $n-1$. 

For 
tetragonal curves we 
show that $S$ can be obtained as 
a rational surface such 
that the $g^{1}_{4}$ is
cut by a pencil of conics; 
more precisely,

\begin{tb}
If a canonical curve $C$ has a $g^1_4$, then 
$F_{\eta_{1},\eta_{2}}$ is a sum of 
cubes of at most $\lceil\frac{3}{2}g-\frac{7}{2}\rceil$  
linear forms, where 
$\eta_1,\eta_2\in H^{0}(C,\omega_{C})$ are general. 
\end{tb}

We show that the above bound is sharp for small genera and it is always better than the 
one in \cite{IR}; we analyse in particular the case of genus $7$ (Subsection~\ref{g7}), showing that, for 
the general tetragonal curve $[C]\in \mathcal{T}_7$ 
($\mathcal{T}_7 \subset\mathcal M_7$ 
is the \emph{tetragonal locus}, which is irreducible 
of dimension $17$), the corresponding $F_{\eta_1,\eta_2}$ is the sum of exactly 
$7$ cubes, and, for $[C]\in \mathcal{T}_7^{(3)}$ 
($\mathcal{T}_7^{(3)}\subset \mathcal{T}_7$  
is the locus of the curves  carrying exactly $3$ linear series $g^1_4$'s; it has dimension $16$), 
$F_{\eta_1,\eta_2}$ is the sum of exactly 
$6$ cubes, in Proposition~\ref{prop:7}. 

Then, in Subsection~\ref{ssec:stc}, 
we have extended the method we used in Subsection~\ref{g7} to some classes of tetragonal curves contained in 
\emph{balanced} rational normal scrolls, see 
Propositions~\ref{4gon}, \ref{4gon-1}, and \ref{4gon-2}, finding---for these curves---better estimates 
than the one in Theorem~B. We have been able to extend the
method to produce surfaces of the desired degree also for curves contained in non-balanced scrolls, 
see Proposition~\ref{specis}; 
but the essential result given in Lemma~\ref{lem:acm} is
not easy to generalise to the obtained surfaces.

Then, in Subsection~\ref{sec:agg}, we show that this construction can be
extended to a class of special curves of every gonality thanks to a referee's suggestion.
Nevertheless the degree of such surfaces is, in general,
rather high---at least higher than the one of 
Iliev and Ranestad (or of Ciliberto and Harris \cite{ch}).

We could not find a generalisation of Theorem~B to higher gonality, since in the proof of it  
we used the fact that the rational normal scroll $X$ which contains the tetragonal 
$C$ has dimension three, and therefore the surface $S$ such that $C\subset S\subset X$ gives a divisor in $X$, 
while for higher gonality the rational normal scroll $X$ 
has higher dimension (more precisely, $\dim(X)=n-1$ for the $n$-gonal 
curve), and therefore $S$ has higher codimension.  

Moreover, we observe that the obstruction to obtain 
the vice versa of Theorem B (or even for the class of curves that we have 
studied) is 
that the geometry of surfaces in $\mathbb P^{g-1}$ of degree $> g$ is
not well understood.


We think that the following problem has its own interest: 
\begin{pb}
Find a bound for the Waring number of the polar hypercubic associated to a general $n$-gonal curve. 
\end{pb}

\begin{acks}
We would like to thank K. Ranestad, E. Mezzetti, G. Sacchiero and M. Brundu for interesting discussions and suggestions, 
G. Casnati and E. Ballico for the help to improve our work and
for pointing out some inaccuracies and valuable comments.  
We would like also to thank the anonymous referee for important remarks and advices. 
\end{acks}

\section{Apolarity and hypercubics}



\subsection{Apolarity}
Let $S:=\CC[x_0,\dotsc, x_N]$ be the polynomial ring in $(N+1)$-variables. 
The algebra of the partial derivatives on $S$, 
\begin{equation*}
T:=\CC[\de_0,\dotsc,\de_N],\qquad \de_i:=\frac{\de}{\de_{x_i}},
\end{equation*}
acts on the monomials in the following way
\begin{equation*}
\de^a \cdot x^b=
\begin{cases}a!\binom{b}{a}x^{b-a} & \text{if $b\ge a$}\\
0 & \text{otherwise}
\end{cases}
\end{equation*}
where $a,b$ are multiindices $\binom{b}{a}=\prod_i\binom{a_i}{b_i}$, etc. 

Obviously, we can think of $S$ as the algebra of partial derivatives on $T$ 
by defining
\begin{equation*}
x^a \cdot \de^b=
\begin{cases}a!\binom{b}{a}\de^{b-a} & \text{if $b\ge a$}\\
0 & \text{otherwise.}
\end{cases}
\end{equation*}

These actions define a perfect paring between the 
homogeneous 
forms in degree $d$ in $S$ 
and $T$:  
\begin{equation*}
S_d\times T_d\xrightarrow{\cdot} \CC.
\end{equation*}
Indeed, this is nothing but the extension of the duality between vector spaces: 
if $V:=S_1$, then $T_1=V^*$. 

Moreover, the perfect 
paring shows the natural duality between $\p^N:=\proj(S)$ and 
$\check\p^N=\proj(T)$. 
More precisely, if $c=: (c_0,\dotsc, c_N)\in\check\p^N$, this gives 
$f_c:=\sum_i c_i x_i\in S_1$, and if $D\in T_a$, 
\begin{equation*}
D \cdot f_c^b=
\begin{cases}a!\binom{b}{a}D(c)f_c^{b-a} & \text{if $b\ge a$}\\
0 & \text{otherwise}.
\end{cases}
\end{equation*}
in particular, if $b\ge a$
\begin{equation*}\label{eq:lin}
0=D \cdot f_c^b \iff D(c)=0.
\end{equation*}

\begin{df}
We say that two forms, $f\in S$ and $g\in T$ are \emph{apolar} if 
\begin{equation*}
g \cdot f =f\cdot g =0.
\end{equation*}
\end{df}

Let $f\in S_d$ and $F:=V(f)\subset\p^N$ the corresponding hypersurface; 
let us now define 
\begin{equation*}
F^\perp :=\{D\in T\mid D\cdot f=0 \}
\end{equation*}
and
\begin{equation*}
A^F:=\frac{T}{F^\perp}.
\end{equation*}

\begin{lem}
The ring $A^F$ is Artinian Gorenstein of socle 
of 
dimension one and degree $d$. 
\end{lem}
\begin{proof}
See \cite[\S 2.3 page 67]{ik}.
\end{proof}

\begin{df}
$A^F$ is called the \emph{apolar} Artinian Gorenstein ring of $F$.
\end{df}

It holds the \emph{Macaulay Lemma}, that is 
\begin{lem}
The map 
\begin{equation*}
F\mapsto A^F
\end{equation*}
is a bijection between the hypersurfaces $F\subset\p^N$ of degree 
$d$ and graded Artinian Gorenstein quotient rings 
\begin{equation*}
A:=\frac{T}{I}
\end{equation*}
of $T$ with socledegree $d$. 
\end{lem}
\begin{proof}
See \cite[Lemma 2.12 page 67]{ik}.
\end{proof}

\subsubsection{Varieties of sum of powers}

Consider a hypersurface $F=V(f)\subset\p^N$ of degree $d$. 
\begin{df}
A subscheme $\Gamma\subset\check\p^N$ is said to be \emph{apolar to $F$} if
\begin{equation*}
\II(\Gamma)\subset F^\perp.
\end{equation*}
\end{df}

It holds the \emph{Apolarity Lemma}:

\begin{lem}\label{lem:apo}
Let us consider the linear forms $\ell_1,\dotsc,\ell_s\in S_1$ and 
let us denote by 
$L_1,\dotsc,L_s\in\check\p^N$ the corresponding points in the dual space. 
Then 
\begin{equation*}
\Gamma\ \textup{is apolar to}\ 
F 
= 
V(f),
\iff \exists \lambda_1,\dotsc,\lambda_s \in \CC^*\ \textup{such that}\ f=\lambda_1\ell_1^d+\dotsc+\lambda_s\ell_s^d
\end{equation*}
where $\Gamma:=\{L_1,\dotsc,L_s\}\subset \check\p^N$. If $s$ is minimal, then it is called the 
\emph{Waring number} of $F$. 
\end{lem}

\begin{proof}
See \cite[Lemma 1.15 page 12]{ik}.
\end{proof}

By this lemma, it is natural to define the \emph{variety of apolar subschemes}
\begin{equation*}
\VPS(F,s):=\overline{\{\Gamma\in\Hilb_s(\check\p^N)\mid 
\II(\Gamma)\subset F^\perp \}},
\end{equation*}
where $\Hilb_s(\check\p^N)$ is the Hilbert scheme of length $s$ 
zero-dimensional subschemes in $\check\p^N$.


\subsection{Hypercubics and canonical sections}
Let $C\subset\p(H^0(\omega_C)^*)=\check\p^{g-1}$ be a canonical curve. 
It is a well-known fact that $C$ is \emph{arithmetically Gorenstein} 
\ie its homogeneous coordinate ring, $\R_C$, is Gorenstein. 
Therefore, if we take two general linear forms 
$\eta_1,\eta_2\in ({\R_C})_1=H^0(\omega_C)$, then 
\begin{align*}
T:&= \frac{\R_C}{\gen{\eta_1,\eta_2}}\\
&= \Sym\left(\frac{H^0(\omega_C)}{\gen{\eta_1,\eta_2}}\right)^*
\end{align*}
is Artinian Gorenstein, and its values of the Hilbert function are 
$1, g-2, g-2, 1$. Therefore, the socledegree of $T$ is 
$3$, and by the Macaulay Lemma, this defines a hypercubic in 
$\proj(S)=\p\left(
\frac{ H^0(\omega_C)}{\gen{\eta_1,\eta_2}}\right)$, $S:=T^*$. 

Thus, we have the rational map 
$\alpha_C\colon\G(g-2,g)\dashrightarrow H_{g-3,3}$ introduced in \eqref{eq:ac}.



\subsubsection{Gonality}

In the following sections we will study the image of the map 
$\alpha_C$.
We will show that this 
is related to study the gonality of $C$.

\section{Trigonal curves} 

In this section we will prove Theorem A:

\begin{thm}
Let $C\subset\p(H^0(\omega_C)^*)$ be a canonical curve. 
Then $C$ is trigonal or isomorphic to a smooth plane quintic 
if and only if for general $\eta_1,\eta_2\in H^0(\omega_C)$, 
the image of the map 
$f=
\alpha_C\left(\frac{H^0(\omega_C)}{\gen{\eta_1,\eta_2}}\right)^*$, 
defined in \eqref{eq:ac}, is a \emph{Fermat cubic}, 
\ie it is the sum of $g-2$ cubes. 
\end{thm}

\begin{proof}
Assume that $C$ is trigonal or isomorphic to a smooth plane quintic. 
Then $\II(C)$ is not generated by 
quadrics by the Enriques-Petri Theorem (see for instance \cite{ACGH}). 
In particular---again by Enriques-Petri---the quadrics determine a surface $S$, 
which is a rational normal scroll 
(or the Veronese surface, in the case of the plane quintic). 
Then, $\II(S)\subset\II(C)$ and $\II(S)_2=\II(C)_2$. 
This implies that the ideal $\II:=(\II(S),\eta_1,\eta_2)$ gives a 
zero-dimensional scheme $\Gamma$ 
of length the degree of $S$ and $\II(\Gamma)=\II$ since $S$ is arithmetically Cohen-Macaulay. Since $S$ is a surface of 
minimal degree in $\check \p^{g-1}$, then $\deg(S)=g-2$. 

By hypothesis, $\eta_1,\eta_2$
are general, then $\II$ gives $g-2$ points in $\check \p^{g-3}$, and by Lemma~\ref{lem:apo}, 
these determine $g-2$ linear forms in $\p^{g-3}$, 
$\ell_1,\dotsc,\ell_{g-2}$ 
such that
$f=\lambda_1\ell_1^3+\dotsc+\lambda_{g-2}\ell_s^3$.

Conversely, if the image of $\alpha_C$ is a Fermat cubic for 
a particular choice of $\eta_1,\eta_2\in H^0(\omega_C)$, 
\ie for 
\begin{equation*}
\left(\frac{H^0(\omega_C)}{\gen{\eta_1,\eta_2}}\right)^*\in \G(g-2,g),
\end{equation*} 
we can fix coordinates
$(\de_0,\dotsc,\de_{g-3})$ on 
$\check\p^{g-3}=\proj(T)
=\p\left(\frac{H^0(\omega_C)}{\gen{\eta_1,\eta_2}}\right)^*$ 
and so coordinates 
$(x_0,\dotsc,x_{g-3})$ on $\p^{g-3}=\proj(S)
=\p\left(\frac{H^0(\omega_C)}{\gen{\eta_1,\eta_2}}\right)$; therefore let us 
suppose that 
\begin{equation*}\label{eq:F}
\alpha_C\left(\frac{H^0(\omega_C)}{\gen{\eta_1,\eta_2}}\right)^*=
[x_0^3+\dotsb+x_{g-3}^3]. 
\end{equation*}
We can also suppose that the coordinates on the projective space $\p(H^0(\omega_C)^*)$ are 
$(\de_0,\dotsc,\de_{g-3},\de_{g-2},\de_{g-1})$, \ie we can think of $\eta_1$ and 
$\eta_2$ as the hyperplanes $\{\de_{g-2}=0\}$ and $\{\de_{g-1}=0\}$, 
respectively.

Letting $f:=x_0^3+\dotsb+x_{g-3}^3$, we only need to find $f^\perp$. 
It is easy to see that 
\begin{equation*}
f^\perp=(\de_i\de_j, \de_i^3-\de_j^3),
\quad i,j\in\{0,\dotsc,g-3\},\quad i\neq j. 
\end{equation*}
Then, the quadrics of $\II(C)$ are of the form 
\begin{equation*}
Q_{i,j}:=\de_i\de_j+\de_{g-2}L_{i,j}+\de_{g-1}M_{i,j}, 
\end{equation*}
where $L_{i,j}$ and $M_{i,j}$ are linear forms on $\check\p^{g-1}$. 
By the Enriques-Petri Theorem, 
$(Q_{i,j})\subsetneq \II(C)$ if and 
only if $C$ is trigonal or isomorphic to a smooth plane quintic. Now, 
$(\de_i\de_j)\subsetneq f^\perp$, since for example 
$\de_0^3-\de_1^3$ is not contained 
in the ideal $(\de_0\de_1)$,
so $(Q_{i,j})\subsetneq \II(C)$: in fact, if it were $(Q_{i,j})= \II(C)$ this would imply 
$(\de_i\de_j)= f^\perp$.

We have just proven that, if for a particular choice of $\eta_1,\eta_2$ the 
image of $\alpha_C$ is a Fermat cubic, then $C$ is trigonal or isomorphic to 
the plane quintic, while we have seen before that if $C$ is trigonal or isomorphic to 
the plane quintic for general $\eta_1,\eta_2$ the image of 
$\alpha_C$ is a Fermat cubic. Then, if $\eta_1,\eta_2$ are general, $f=
\alpha_C\left(\frac{H^0(\omega_C)}{\gen{\eta_1,\eta_2}}\right)^*$, is a Fermat cubic. 
\end{proof}


\section{The tetragonal case}

In order to analyse $F_{\eta_{1},\eta_{2}}$ when $C$ is an $n$-gonal curve, 
we recall some basic well-known facts about rational 
normal scrolls.

\subsection{Rational normal scrolls}
By definition, a \emph{rational normal scroll} (RNS for short, in the following) 
of type $(a_1,\dotsc,a_k)$, $S_{a_1,\dotsc,a_k}$, is the 
image of the $\p^{k-1}$-bundle $\p(E)=\p(\OO_{\p^1}(a_1)\oplus\dotsb\oplus\OO_{\p^1}(a_k))$, 
$\pi\colon\p(E)\to\p^1$, via 
the map $j$ given by $\OO_{\p(E)}(1)$ in $\p^N$, $N=\sum a_i+k-1$. 
Equivalently, one takes $k$ disjoint 
projective spaces of dimension $a_i$, $\p^{a_i}$, and $k$ rational normal curves $C_i\subset\p^{a_i}$, 
together with isomorphisms $\phi_i\colon\p^1\to C_i$ (if $a_i\neq 0$, constant maps otherwise); 
then
\begin{equation*}
S_{a_1,\dotsc,a_k}=\bigcup_{P\in\p^1}\gen{\phi_1(P),\dotsc,\phi_k(P)}.
\end{equation*}
We have also that
\begin{align*}
\deg (S_{a_1,\dotsc,a_k}) &=\sum a_i\\
&=N-k+1,
\end{align*}
and $\dim (S_{a_1,\dotsc,a_k})=k$. 
From the second description, 
it is an easy exercise to show that, if $P\in C_i$, then the projection of $S_{a_1,\dotsc,a_k}$ from 
$P$ is a rational normal scroll of type $(a_1,\dotsc,a_{i-1},a_i-1,a_{i+1},\dotsc,a_k)$, with 
the convention $$(a_1,\dotsc,a_{i-1},-1,a_{i+1},\dotsc,a_k)=(a_1,\dotsc,a_{i-1},a_{i+1},\dotsc,a_k).$$

We note that $j\colon \p(E) \to S_{a_1,\dotsc,a_k} \subset\p^N$ is an isomorphism (and $S_{a_1,\dotsc,a_k}$ 
is smooth) if (and only if) $a_i>0$, for all $i$. 

The \emph{Picard group} of $\p(E)$ is generated by the \emph{hyperplane class}, defined by $H:=[j^*\OO_{\p^N}(1)]$, 
and by its \emph{ruling}, \ie $F:=[\pi^*\OO_{\p^1}(1)]$, and the intersection product is given by  
\begin{equation*}
H^k=N-k+1, \qquad H^{k-1}F=1, \qquad F^2=0. 
\end{equation*}
Finally, we recall that the canonical class of $\p(E)$ is 
\begin{equation*}
K_{\p(E)}=-kH+(N-k-1)F. 
\end{equation*}

The following well-known theorem (due to A. Maroni, \cite{Ma}) relates the $n$-gonal curves with the RNS: 

\begin{thm}\label{thm:sac}
Let $C\subset\p(H^0(\omega_C)^*)$ be a canonical curve and $n$ is an integer $n\ge 4$. Then $C$ 
has a $g^1_n$ if and only 
if it is contained in a 
rational normal scroll of dimension $n-1$ (and so 
of degree $g-n+1$). Moreover, the $\check\p^{n-2}$'s which are the fibres of the 
scroll, cut on $C$ precisely the $g^1_n$. 
\end{thm}

\begin{rmk}
We put $n\ge 4$ only to avoid the case of the plane quintic, which has a 
$g^2_5$ instead of a $g^1_3$. Of course it has a $g^1_4$ and it is contained 
in a rational normal cubic threefold in $\check\p^5$. 
\end{rmk}

\begin{proof}
By the geometric version of the Riemann-Roch Theorem, a divisor $D\in g^1_n$ 
generates a $\check\p^{n-2}$ in $\check\p^{g-1}$. The union of these $\check\p^{n-2}$'s
generates 
a rational scroll $X$ of dimension $n-1$. Therefore it is a RNS or 
a projection of a RNS. 
It is not a projection since $C$ is projectively normal (and hence linearly normal).

For the vice versa, if $C$ is contained in the scroll $X$, then the $\check\p^{n-2}$'s of $X$ 
determine a linear system $\abs{D}$ of dimension at least one; 
then, again by the geometric version of the Riemann-Roch Theorem,
$\deg(D)\le n$, and the theorem is proved. 

\end{proof}

\begin{rmk}
A classical result (due to B. Segre, see \cite{BSeg}) assure us that the above $\check\p^{n-2}$'s in the proof of the 
theorem 
are in general positions for the general curve, and therefore $X$ is smooth if the curve $C$ is general. 
\end{rmk}

To ease exposition, we give the following: 

\begin{df}
We say that a RNS $X$ as in Theorem~\ref{thm:sac} is a \emph{balanced scroll}, if 
$X=\p(\OO_{\p^1}(a_1)\oplus\dotsb\oplus\OO_{\p^1}(a_{n-1}))$ is 
such that 
$a=\lfloor\frac{g}{n-1}\rfloor-1$, where $a:=\min_{i\in\{1,\dotsc,n-1\}}\{a_i\}$. 
\end{df}

By a moduli computation 
we note that a general $n$-gonal curve is contained in a balanced scroll (see for example \cite{DG}, \cite{Cha}, or \cite[Corollary 3.3]{GV}).


\subsection{Theorem B}
We analyse a canonical curve 
$C\subset\p(H^0(\omega_C)^*)$ having a $g^1_4$. 
Here, the intersection of the scroll $X$, given by Theorem~\ref{thm:sac}, 
with the $\check\p^{g-3}$, is 
no longer given by a finite number of points, but it is a (rational normal) 
curve. 
Then, we want to find a surface $S$ such that $C\subset S\subset X$, and if $\eta_1$ and $\eta_2$ are 
two $1$-forms on $C$, then, letting $\Gamma:=V(\II(S),\eta_1,\eta_2)$, it holds that 
$\II(\Gamma)$ is contained in $(\II(C),\eta_1,\eta_2)=F^\perp$. 

\begin{lem}\label{lem:acm}
Let $\phi_{\abs{K_S+C}}\colon S\to \p^{g-1}$ be a generically 1--1 map, where $S$ is a rational (smooth) surface and 
$C$ is a smooth curve of genus $g\ge 4$. Let us call $Y(\subset \p^{g-1})$ the image of this map. 
Moreover, let $\eta_1,\eta_2\in H^0(\OO_{\p^{g-1}}(1))$ be general sections and $\Gamma:=V(\II(Y),\eta_1,\eta_2)$ and  
finally, we suppose that $(\II(C),\eta_1,\eta_2)=F^\perp$, where 
$F$ is a cubic polynomial in the dual coordinates. Then 
$\II(\Gamma)\subset (\II(C),\eta_1,\eta_2)$. 
\end{lem}

\begin{proof}
 Since $\II(\Gamma)$ is generated in degree $2$ and $3$,
 by the hypothesis on $F$, it is sufficient to prove that 
$H^0(\p^{g-1},\II_{\Gamma/\p^{g-1}}(j))\subset (F^\perp)_j$, $j=2,3$. 

Let $D_1, D_2\in \abs{K_S+C}$ be such that $D_i:=\{\phi^*(\eta_i) =0)\}$, $i=1,2$. 
By the cohomology of the standard resolution of 
$\II_{\Gamma/S}:=\phi^*(\II_{\Gamma/Y})$, we obtain that 
$H^0(S,\II_{\Gamma/S}(2))=\phi^*(\eta_1) H^0(S,K_S+C)+ \phi^*(\eta_2) H^0(S,K_S+C)$ since $S$ is 
regular. 

Since $H^1(S,K_S+C)=0$, then $H^0(S,\II_{\Gamma/S}(3))=\phi^*(\eta_1) H^0(S,2(K_S+C))+ \phi^*(\eta_2) H^0(S,2(K_S+C))$. 

Now, consider the standard exact sequence of ideal sheaves:
\begin{equation*}
0\to\II_{Y/\p^{g-1}}(j)\to \II_{\Gamma/\p^{g-1}}(j)\to \II_{\Gamma/Y}(j)\to 0; 
\end{equation*}
Let $Q\in H^0(\p^{g-1},\II_{\Gamma/\p^{g-1}}(j))$, $j=2,3$. If $Q\in H^0(\p^{g-1},\II_{Y/\p^{g-1}}(j))$, we are done; otherwise, we obtain 
a $q\in H^0(\p^{g-1},\II_{\Gamma/Y}(j))\cong H^0(S,\II_{\Gamma/S}(j))$, and we can conclude by what we have proven above.

\end{proof}

\begin{rmk}Notice that the conclusion of Lemma~\ref{lem:acm} holds if the surface $Y$ 
is arithmetically Cohen-Macaulay. In the paper \cite{IR}, they use 
the fact that their surfaces are cones over arithmetically 
Cohen-Macaulay curves, hence any general linear section of these 
cones is aCM.
\end{rmk}

Let us now prove Theorem B:

\begin{thm}\label{thm:tetra}
If a canonical curve $C\subset\check\p^{g-1}$ has a $g^1_4$, then 
$F_{\eta_{1},\eta_{2}}$ is a sum of 
cubes of at most $\lceil\frac{3}{2}g-\frac{7}{2}\rceil$  
linear forms, where 
$\eta_1,\eta_2\in H^{0}(C,\omega_{C})$ are general $1$-forms. 
\end{thm}

\begin{proof}
By Theorem~\ref{thm:sac}, $C\subset X$, where $X\subset\check\p^{g-1}$ is a rational normal smooth 
threefold of degree $g-3$.  
 In the Chow ring of $X$ the curve $C$ is 
\begin{equation*}
[C]\sim 4 H^2+2(5-g)HF 
\end{equation*}
(see \cite[ Theorem 3.1]{GV}). By \cite[Corollary~(4.4)]{Sch} 
(see also \S 6 there), $C$ is 
the complete intersection of two irreducible surfaces of type
\begin{equation*}
[Y_1]\sim 2H-b_1 F,\ [Y_2]\sim 2H-b_2 F, \qquad \textup{with}\ b_1+b_2=g-5.  
\end{equation*}
It follows: 
\begin{align}
\deg Y_i&=(2H-b_i F)\cdot H^2\\
&=2g-6-b_i,\label{bi}. 
\end{align}
Therefore, if $b_1\ge b_2$, then $\deg(Y_2)\ge\deg(Y_1)$
and, since $b_2\ge \lfloor\frac{g-5}{2}\rfloor$,
we deduce 
\begin{align*}
\deg(Y_2)&\le 2g-6-\left \lfloor \frac{g-5}{2}\right\rfloor\\
&= \left\lceil \frac{3g-7}{2}\right\rceil.
\end{align*}

We note that $Y_{i}$ restricts to a quadric 
on each fibre $\Pi$ of the scroll $X$. 
The intersection $Y_{1}\cap Y_{2}\cap\Pi$ gives the four points 
of $C$ which are a divisor of the $g^1_4$. Again by \cite[\S 6]{Sch} 
at least $Y_{2}\cap \Pi$ is generically irreducible.  
Moreover, $Y_2$ satisfies the hypothesis of 
Lemma~\ref{lem:acm} (see for example \S 5 of \cite{Sch}). Then, the 
theorem is proved. 
\end{proof}

\begin{rmk}\label{rmk:sch}
A more precise analysis of the proof of the preceding Theorem can be done if we distinguish the cases in which $g$ is 
odd or even: if $g=2k+1$ is odd, it is immediate that the degree of both $Y_1$ and $Y_2$ is bounded by 
$2g-6-(k-2)=3k-2=\frac{3g-7}{2}$; instead, if $g=2k$ is even, we deduce that the degree of $Y_1$ is bounded by 
$2g-6-(k-2)=3k-4=\frac{3}{2}g-4$, which is strictly less than $2g-6-(k-3)=3k-3=\frac{3}{2}g-3$, the bound for $Y_2$. 
The problem is that in this situation $Y_1$ could not be rational: in this case the $g^1_4$ is composed by an elliptic or 
hyperelliptic involution: see \cite[\S 6.5]{Sch}. 
\end{rmk}

The surfaces $Y_i$ of the proof of the above theorem has been studied before (independently to us) 
and more deeply in 
\cite{bs} to give a stratification of the moduli space of the tetragonal curves.


\subsection{Low genus cases}

We recall that every curve of genus $g\le 5$ has a $g^1_4$ (see  \cite[IV.5.5.1]{H}); 
so let us start to analyse the cases of low genus:

\subsubsection{g=6}
In this case, $C\subset\check\p^5$ and $\deg(C)=10$. If 
$C$ is tetragonal, it is contained in a smooth cubic threefold and we 
find a sharp estimate: 
Theorem~\ref{thm:tetra} says that there is a surface of 
degree at most $6$ containing $C$ and contained in $X$. 
Indeed, from Remark~\ref{rmk:sch} we can see that $C$ 
is contained in a quartic surface; but this is a surface of 
minimal degree, and it is either a RNS, in which case 
$C$ is trigonal, or it is a Veronese surface, in 
which case the $g^1_4$ is induced by the $g^2_5$
which corresponds to the conics of the Veronese surface, in accordance 
with Enriques-Petri Theorem.

\subsubsection{g=7}\label{g7}
In this case, $C\subset\check\p^6$ and $\deg(C)=12$; if $C$ is tetragonal it is contained in a smooth quartic threefold.
Theorem~\ref{thm:tetra} says that there is a surface of degree at most $7$ containing $C$ and contained in $X$. 
From---for example---\cite[\S 6]{Mu}, 
we see that $7$ is the correct estimate for a general tetragonal 
curve. But it is also shown, again in \cite[\S 6]{Mu}, 
that for the special ones, \ie if $C$ possesses a $g^2_6$, 
then $C$ is contained in a 
(possibly singular) \emph{sextic del Pezzo} surface. 
We give here an alternative proof of the existence of this surface. 

First, we note that the quartic scroll $X$  which contains $C$ is of type 
$X=\p(\OO_{\p^1}(1)\oplus\OO_{\p^1}(1)\oplus\OO_{\p^1}(2))$.
Let us consider our surface $S$ as in the proof of 
Theorem~\ref{thm:tetra}, \ie $[S]=2H-bF$ as a divisor on $X$.

Then,  
we project the $C$ to $\check\p^2$ by choosing the centre of projection 
$\pi\colon\check\p^6\dashrightarrow\check\p^2$ a $\check\p^3$
generated by a general plane of the scroll and a general point $P$ of the curve. 
So, 
we obtain a singular 
plane curve $Z$ of degree $7$. 
Let $V_i$'s be the singular points of $Z$. We can suppose that the 
points are of simple multiplicity $m_i$.

Let us now \emph{suppose} that the surface $S$ containing $C$ is  
the blowing-up of $\check\p^2$ in the points $V_i$'s: $\pi_{\{V_i\}}\colon S \to \check\p^2$. 
This means that the projection 
\begin{equation*}
\pi\mid_{S}\colon S\dashrightarrow \check\p^2
\end{equation*}
is generically 1--1. 

Let us denote by $H:=\pi_{\{V_i\}}^*\OO_{\check\p^2}(1)$ the hyperplane section of $S$ 
and by $E_i$'s the $(-1)$-curves on $S$ which correspond to the $V_i$'s. 
The complete linear system $\abs{4H+\sum_i(1-m_i)E_i}$ gives a generically 1--1 map 
$\phi\colon\ S\to\check\p^6$ such that $C=\phi(\pi_{\{V_i\}}^{-1}(Z))$. 
In fact, we can write, by adjunction
\begin{align*}
C&\in\abs{7H-\sum_i m_i E_i}\\
K_S&=-3H+\sum_i E_i\\
\omega_C&=(4H+\sum_i(1-m_i)E_i)_{\mid_C}.
\end{align*}
Then, the adjunction formula of $C$ on $S$ yields
\begin{equation*}
(7H-\sum_i m_iE_i)\cdot (4H+\sum_i(1-m_i)E_i)=12
\end{equation*}
which means, 
\begin{equation}\label{adj4}
\sum_i m_i(m_i-1)=16. 
\end{equation}

If we take a general plane $\Pi_t$, $t\in\p^1$ of the ruling of the scroll 
$\rho\colon X\to\p^1$, then $\pi\mid_{\Pi_t}\colon\Pi_t\to\check\p^2$ 
is a birational isomorphism. Now, through the 
four points of $C\cap \Pi_t$ there passes a pencil of conics $\Lambda_{\Pi_t}$. 
Denote by $Q_{\Pi_t}$ the generic conic of $\Lambda_{\Pi_t}$. 
Let us denote by 
$Q_{\Pi_t}':=\pi(Q_{\Pi_t})$. We want to show that there exists a pencil 
$\Lambda:=\{Q_t'\}_{t\in\p^1}$ in $\check\p^2$ such that $Q_t'$ comes from 
a general specialisation of $Q_{\Pi_t}$. In fact, if we consider the projection of $X$ from $P$, which we can 
suppose, by its generality, it is a general point of $C_2$, \ie 
a unisecant conic of the scroll, 
we obtain that $X$ is mapped to a balanced scroll $\p(\OO_{\p^1}(1)\oplus\OO_{\p^1}(1)\oplus\OO_{\p^1}(1))$, and 
one unisecant line gives in fact a fifth point in each plane which determines a conic in each plane of the ruling and 
then a pencil of conics when we project further to $\check\p^2$. 

So, $\Lambda$ cuts the $g^1_4$ on $Z$. Let $A_1,\dotsc,A_4$ be the base points of the
pencil; by construction, we have
\begin{equation*}
Q'_t\mid_Z=\sum_{i=1}^4 (n_iA_i+P'_{it}),
\end{equation*}
where $\{P_{1t},\dotsb,P_{4t}\}=\Pi_t\cap C$ and $P'_{it}=\pi(P_{it})$. 
Without loss of generality, by Bez\'out, we can assume $V_i=A_i$, and by the generality of $Q'_t$, $n_i=m_i$. 
But then
\begin{align*}
\deg Q'_t\mid_Z &=2\deg Z\\
&=14\\
&=4+\sum_{i=1}^4 n_i, 
\end{align*}
which means 
\begin{equation}\label{qat}
\sum_{i=1}^4 n_i=10,
\end{equation}
from which we deduce 
\begin{equation*}
\sum_{i=1}^4 n_i^2\ge 25,
\end{equation*}
which implies 
\begin{equation*}
\sum_{i=1}^4 n_i^2-n_i\ge 15. 
\end{equation*}
Now, $\sum_{i=1}^rm_i^2-m_i=16>15$, therefore $r=4$, and from \eqref{qat}, we deduce 
\begin{equation}\label{3k-34}
\sum_{i=1}^4 (m_i-1)=6. 
\end{equation}
From this, 
\begin{align*}
\deg(S)&=(4H+\sum_i(1-m_i)E_i)^2\\
&= 16-\sum_i(m_i-1)^2,
\end{align*}
but then, by \eqref{adj4} and \eqref{3k-34}
\begin{align*}
\sum_i(m_i-1)^2&=\sum_i (m_i(m_i-1)-(m_i-1))\\
&=16-6\\
&=10
\end{align*}
we obtain that 
\begin{equation*}
\deg(S)=6. 
\end{equation*}
Therefore having assumed that $S$ is the blowing-up of 
$\check\p^{2}$, we have that $\deg(S)=6$. In particular 
$[S]=2H-2F$ as a divisor on $X$, by \eqref{bi}. 

By the above proof, if the singular points 
$V_1,\dotsc,V_4$ are distinct, then their multiplicities must be 
$m_1=m_2=3$ and $m_3=m_4=2$, and therefore we have three $g^1_4$ on $C$: 
the first one given by the pencil 
of conics $\Lambda$, 
and the other two correspond 
to the pencils of lines through $V_1$ and $V_2$. 

Now we show the main result of this subsection
that the canonical tetragonal curves 
of genus $7$ with a $g^2_6$ can be realised exactly
as the above blowing-up of a curve $Z$. 

Let us see this. 
First of all, we recall that, in 
general, the $n$-gonal locus is irreducible in the (coarse) 
moduli space of algebraic curves 
of genus $g$, $\mathcal{M}_g$; this follows for example from \cite{F}. 

In the case $g=7$, 
much more can be said: see for example \cite[Table 1]{Mu}, 
where it is explained the stratification of this 
moduli space. In particular, inside $\mathcal{M}_7$, 
which has dimension $18$, there is 
the codimension one \emph{tetragonal locus},  $\mathcal{T}_7$. 
The general curve in it has only one $g^1_4$. 
Inside the tetragonal locus, there is 
the locus of curves possessing a $g^2_6$, $\mathcal{G}^2_6$. 
$\mathcal{G}^2_6$ has codimension one in $\mathcal{T}_7$ so, 
$\dim\mathcal{G}^2_6=16$. 
The general curve in $\mathcal{G}^2_6$ has exactly two $g^2_6$. 
Moreover it can have one, two or three $g^1_4$.

This case has been analysed 
deeply in \cite{Cas}. First of all, one can observe that 
$\mathcal{T}_7=\mathcal{T}_7^{(1)}\cup\mathcal{T}_7^{(2)}\cup\mathcal{T}_7^{(3)}\cup\mathcal{M}_7^{be}$, where 
$\mathcal{T}_7^{(i)}$  is the locus of the curves  carrying exactly $i$ linear series $g^1_4$'s, and 
$\mathcal{M}_7^{be}$ is the bielliptic locus. 
One of the main results of  \cite{Cas} is that the loci  $\mathcal{T}_7^{(2)}$ and $\mathcal{T}_7^{(3)}$ are 
(irreducible) rational subvarieties of dimensions $15$ and $16$, respectively, of $\mathcal{M}_7$. 
In particular, the general element of $\mathcal{G}^2_6$ 
is in $\mathcal{T}_7^{(3)}$, \ie it has three $g^1_4$'s. 

Precisely, in \cite[\S 2]{Cas} it is shown why $\mathcal{T}_7^{(3)}$ is rational of dimension $16$, with the 
following geometric 
argument. Let $C$ be our general tetragonal curve of genus $g = 7$ carrying three $g^1_4$'s. 
Then, $C$ has a sextic plane model $\tilde C \subset\check\p^3$ with three non–collinear nodes, which can be assumed 
to be $P_1=(1:0:0)$, $P_2=(0:1:0)$ and $P_3=(0:0:1)$. Clearly, such a model depends on the choice of the $g^2_6$.
Set $X := \{P_{1} , P_{2} , P_{3} \}$. If $\varphi\colon C \to C'$ is an isomorphism, 
then, in particular, it sends the $g^r_n$'s on $C$ into $g^r_n$'s on $C'$, thus it induces
a birational automorphism $\phi\in \bir_X (\check\p^2 )$, defined on the whole of $\check\p^2 \setminus X$, leaving $X$ fixed 
and sending $\tilde C$ to $\tilde C'$. 
$\bir_X(\check\p^2)$ is generated by the torus of diagonal matrices $\p T\subset \p \Gl(3)$, with $\dim(\p T)=2$, 
by the standard quadratic Cremona transformation $\mu(x:y:z) = (yz:xz:xy)$, which 
permutes the two $g^2_6$'s on $C$ and by the group of permutations of the $P_i$'s. 

The subspace $W\subset \CC[x, y, z]$ of forms of degree $6$ representing
plane curves having singularities at the points $P_i$'s has dimension $19$. Moreover, 
the action of $\bir_X(\check\p^2 )$ on $\check\p^2$ induces a linear action on $\abs{W}$, \ie 
$\bir_X (\check\p^2)$ can be realised as a subgroup of $\p \Gl(W )$. 

Consider then the natural map  $p\colon \Gl(W ) \to \p \Gl(W )$ and  let us define   
$G := p^{-1} (\bir_X (\check\p^2 )) \subset\Gl(W )$. Then, they notice that there exists 
a dominant rational map $W\to \mathcal T_7^{(3)}$, whose
fibres are the $G$-orbits of $W$, and therefore 
$\mathcal T_7^{(3)}$ is irreducible of dimension $16$. 

Let us show now that there is a map which associates our $7$-tic $Z\subset\check\p^2$ to one of their sextics $\tilde C$. 
We can suppose that the $P_1=(1:0:0)$ and $P_2=(0:1:0)$ are the triple points of $Z$, and $P_3=(0:0:1)$ and $P_4=(1:1:1)$ 
are its nodes. The two $g^2_6$ on $Z$ are given by the conics passing through $P_3$ and $P_4$, and 
through one of the two double points, \ie $\abs{D_i}:=\abs{2H-P_1-P_2-P_i}$, $i=3,4$. Therefore, a map of the type 
\begin{equation*}
\phi_{D_4}\colon Z\dashrightarrow \tilde C
\end{equation*}
is what we were looking for. In fact, $\tilde C:=\overline{\phi_{D_4}(D_4)}$ is a sextic which has $P_1$, $P_2$ and 
$P_3$ as its singular points, and they are nodes. 

Now, it is not difficult to show that we can come back 
from $\tilde C$ to one of our $Z$. 
Then even our construction 
gives all the curves of genus $7$ with a $g^2_6$.
Instead, we give a computation with the moduli (and this is sufficient, 
since the moduli spaces we are considering are all irreducible). 
As above,  if $\varphi\colon C \to C'$ ($C$ and $C'$ canonical curves, 
normalisation of the two septics $Z$ and $Z'$) 
is an isomorphism, then in particular, it sends the $g^r_n$'s on $C$ 
into $g^r_n$'s on $C'$, and induces 
a birational automorphism $\phi\in \bir_Y (\check\p^2 )$ 
defined in $\check\p^2 \setminus Y$, where $Y:=\{P_1 , P_2 , P_3, P_4 \}$. 

Now, the main difference with the case of the sextics is 
that for $Z$ and $Z'$ the three $g^1_4$ are given by the lines 
through $P_1$ and $P_2$, and the conics through $Y$, 
therefore in $\bir_Y (\check\p^2 )$ we have also the Cremona 
transformations of the plane which send the lines 
through $P_1$ (or $P_2$) to the conics through $Y$. 
These transformations form a group of dimension one, and  $\bir_Y (\check\p^2 )$ is generated by these transformations plus the 
maps which change $P_1$ with $P_2$ and $P_3$ with $P_4$. 

Now, the subspace $W'\subset \CC[x, y, z]$ of forms of degree $7$ 
representing
plane curves having singularities at the points $P_i$'s has dimension 
$9\cdot 4-2\cdot(3+6)=18$; therefore the family of 
curves we have found has dimension $16$ in $\mathcal T_7^{(3)}$, and 
therefore it coincides with the whole 
$\mathcal T_7^{(3)}$. 

So, we can summarise what we have proven in the following:

\begin{prop}\label{prop:7}
The general element of $\mathcal T_7^{(3)}$ can be obtained from a (tetragonal)
curve 
$C\subset S_{1,1,2} \subset \check\p^6$ contained in  
a sextic rational surface $S$ of type $[S]=2H-2F$ in the rational normal scroll  
$S_{1,1,2}$. 

The projection $\pi\colon\check\p^6\dashrightarrow\check\p^2$
from the $\check\p^3$ generated by a general plane of  $S_{1,1,2}$ 
and a general point of $C$, 
restricted to $S$, is generically 1--1 (\ie $S$ is the blowing-up of this 
$\check\p^2$). 

The curve $Z:=\pi(C)\subset\check\p^2$ has degree $7$ and has, as singular locus, four points in general position, two of them
are nodes and two are triple points. 

Moreover, for the general tetragonal canonical curve of genus $7$, 
the corresponding cubic $F_{\eta_1,\eta_2}$ is the sum of 
exactly $7$ cubes, while 
if a tetragonal canonical curve of genus $7$ is a 
general element of $\mathcal T_7^{(3)}$ (\ie it has two $g^2_6$'s), 
then  $F_{\eta_1,\eta_2}$ is the sum of exactly $6$ cubes. 
\end{prop}

\begin{proof}
It remains to prove that $F_{\eta_1,\eta_2}$ is the sum of exactly $6$ cubes. 
We refer to \cite{Fu} for general facts about low degree varieties.  
Since $C$ is neither trigonal nor degenerate, 
the degree of such a surface $S$ cannot be $\le 5$, 
and $S$ cannot have sectional geometric 
genus zero (\ie it cannot be a projection of a rational normal surface).  

If $C$ is general, it cannot be contained in a surface of 
degree $6$, since this cannot have sectional geometric genus one; 
in fact, in this case, $S$ is either a Del Pezzo surface, and 
this would imply that there is a $g^2_6$ on $C$, cut out by  
rational normal cubic curves of $S$, 
or a cone over an elliptic curve, and therefore projecting 
from the vertex of the cone, we would see that $C$ is bielliptic.  
\end{proof}


\subsection{Special tetragonal curves}\label{ssec:stc} 

In this subsection we want to generalise 
the construction obtained in Subsection~\ref{g7}.

In the following Propositions~\ref{4gon}, \ref{4gon-1}, and \ref{4gon-2},
the surfaces we consider are as in Lemma~\ref{lem:acm}, so we can 
apply them to the Waring problem.

We study first the case with $g=3k$, since we can find explicitly a rational smooth 
surface $S$ of degree $\frac{4}{3}g-3$, such that $C\subset S\subset X$: 

\begin{prop}\label{4gon}
Let $C\subset\p(H^0(\omega_C)^*)$ be a tetragonal canonical curve
of genus $g=3k$, where $k\geq 2$, contained in a rational surface $S$ of type $[S]=2H-bF$ in a balanced scroll 
$X=\p(\OO_{\p^1}(k-1)\oplus\OO_{\p^1}(k-1)\oplus\OO_{\p^1}(k-1))$. 

Let us suppose that the projection $\pi\colon\check\p^{g-1}\dashrightarrow\check\p^2$
from $k-1$ general planes of the scroll restricted to $S$ is generically 1--1 (\ie $S$ is the blowing-up of this 
$\check\p^2$). 
Then,  
\begin{equation*}
\deg(S)=\frac{4}{3}g-3,  
\end{equation*}
or, equivalently, $b=\frac{2}{3}g-3$. 

In particular, $F_{\eta_1,\eta_2}$ is a sum of 
cubes of at most  $\frac{4}{3}g-3$ 
linear forms, where 
$\eta_1,\eta_2\in H^{0}(C,\omega_{C})$ are general. 
\end{prop}


\begin{proof}

The proof follows the idea behind our construction of 
$\mathcal{T}^{(3)}_{7}$. The image of $C$ under $\pi$ is 
a singular 
plane curve $Z$ of degree $2k+2$. 
Let $V_i$'s be the singular points of $Z$. We can suppose that the 
points are of simple multiplicity $m_i$.

By the hypothesis, we can think of the surface $S$ containing $C$ as 
the blowing-up of $\check\p^2$ in the points $V_i$'s: $\pi_{\{V_i\}}\colon S \to \check\p^2$.
Let us denote by $H:=\pi_{\{V_i\}}^*\OO_{\check\p^2}(1)$ the hyperplane section of $S$ 
and by $E_i$'s the $(-1)$-curves on $S$ which correspond to the $V_i$'s. 
The complete linear system $\abs{(2k-1)H+\sum_i(1-m_i)E_i}$ gives a generically 1--1 map  
$\phi\colon\ S\to\check\p^{g-1}$ such that $C=\phi(\pi_{\{V_i\}}^{-1}(Z))$. 
In fact, we can write, by adjunction
\begin{align*}
C&\in\abs{(2k+2)H-\sum_i m_i E_i}\\
K_S&=-3H+\sum_i E_i\\
\omega_C&=((2k-1)H+\sum_i(1-m_i)E_i)_{\mid_C}.
\end{align*}
Then, the adjunction formula of $C$ on $S$ yields
\begin{equation*}
((2k+2)H-\sum_i m_iE_i)\cdot ((2k-1)H+\sum_i(1-m_i)E_i)=6k-2
\end{equation*}
which means, 
\begin{equation}\label{adj}
\sum_i m_i(m_i-1)=4k(k-1). 
\end{equation}

If we take a general plane $\Pi_t$, $t\in\p^1$ of the ruling of the scroll $\rho\colon X\to\p^1$, then $\pi\mid_{\Pi_t}\colon\Pi_t\to\check\p^2$ 
is an isomorphism. Now, through the 
four points of $C\cap \Pi_t$ there passes a pencil of conics $\Lambda_{\Pi_t}$. Denote by $Q_{\Pi_t}$ the generic conic of $\Lambda_{\Pi_t}$. 
Let us denote by 
$Q_{\Pi_t}':=\pi(Q_{\Pi_t})$. We want to show that there exists a pencil $\Lambda:=\{Q_t'\}_{t\in\p^1}$ in $\check\p^2$ such that $Q_t'$ comes from 
a general specialisation of $Q_{\Pi_t}$. In fact, if we consider the projection of $X$ from $k-2$ planes instead of 
$k-1$, we obtain that $X$ is mapped to a balanced scroll $\p(\OO_{\p^1}(1)\oplus\OO_{\p^1}(1)\oplus\OO_{\p^1}(1))$, and 
one unisecant line gives in fact a fifth point in each plane which determines a conic in each plane of the ruling and 
then a pencil of conics when we project further to $\check\p^2$. 

So, $\Lambda$ cuts the $g^1_4$ on $Z$. Let $A_1,\dotsc,A_4$ be the base points of the
pencil; by construction, we have
\begin{equation*}
Q'_t\mid_Z=\sum_{i=1}^4 (n_iA_i+P'_{it}),
\end{equation*}
where $\{P_{1t},\dotsb,P_{4t}\}=\Pi_t\cap C$ and $P'_{it}=\pi(P_{it})$. 
Without loss of generality, by Bez\'out, we can assume $V_i=A_i$, and by the generality of $Q'_t$, $n_i=m_i$. 
But then
\begin{align*}
\deg Q'_t\mid_Z &=2\deg Z\\
&=4k+4\\
&=4+\sum_{i=1}^4 n_i, 
\end{align*}
which means 
\begin{equation}\label{equat}
\sum_{i=1}^4 n_i=4k,
\end{equation}
from which we deduce 
\begin{equation}\label{ineq}
\sum_{i=1}^4 n_i^2\ge 4k^2,
\end{equation}
and the equality holds iff $n_i=k$, $\forall i$. 
Then 
\begin{equation*}\label{eq:ni}
\sum_{i=1}^4 n_i(n_i-1)\ge 4k(k-1),
\end{equation*}
which implies, by Equation \eqref{adj}, $m_i=n_i=k$ for $1\le i\le 4$ and $m_j=0$ for $j\ge 5$.

Then, 
\begin{equation}\label{3k-3}
\sum_i (m_i-1)=4k-4. 
\end{equation}

Now, 
\begin{align*}
\deg(S)&=((2k-1)H+\sum_i(1-m_i)E_i)^2\\
&= (2k-1)^2-\sum_i(m_i-1)^2,
\end{align*}
but then, by \eqref{adj} and \eqref{3k-3}
\begin{align*}
\sum_i(m_i-1)^2&=\sum_i (m_i(m_i-1)-(m_i-1))\\
&=4k(k-1)-(4k-4)\\
&=4k^2-8k+4
\end{align*}
we obtain that 
\begin{align*}
\deg(S)&=4k-3\\
&=\frac{4}{3}g-3. 
\end{align*}
\end{proof}

\begin{rmk}
We can show that the canonical curves and surfaces of the preceding proposition actually exist, in the following way. 
From the proof of the proposition, and with the same notation, we deduce that $S$ is the image of 
the blowing-up of $\check\p^2$ in the points $P_1,\dotsc,P_4$ by 
the linear system $\abs{(2k-1)H-(k-1)\sum_{i=1}^4P_i}$. 
Now, it is immediate to see that the dimension of $\abs{(2k-1)H-(k-1)\sum_{i=1}^4P_i}$ is at least $3k+1$,
and therefore, 
in order to show that $S$ exists, it is sufficient to show that this linear system  
is ample. This fact can be obtained by the Nakai-Moishezon Criterion, \cite[Theorem V.1.10]{H}: in fact, 
first of all, 
\begin{equation*}
((2k-1)H-(k-1)\sum_{i=1}^4P_i)^2= 4k-3>0. 
\end{equation*}
Then, if $D$ is an irreducible curve in the blowing up of  $\check\p^2$ in the points $P_1,\dotsc,P_4$, 
we can think of it as $D\in\abs{aH-\sum_{i=1}^4b_iP_i}$, 
with $a>0$, and $b_i\ge 0$, $\forall i$; 
if, by contradiction, we suppose that 
$L\cdot D \le 0$, with $L\in\abs{(2k-1)H-(k-1)\sum_{i=1}^4P_i}$, then, we deduce that 
\begin{equation*}
(2k-1)a-k\sum_{i=1}^4b_i\le 0, 
\end{equation*}
which means
\begin{align}
\sum_{i=1}^4b_i &\ge \frac{(2k-2)a+a}{k-1}\\
&= 2a+\frac{a}{k-1}\\
&> 2a. \label{bii}
\end{align}
From this, since we can write $\sum_{i=1}^4b_i>4\frac{a}{2}$, we obtain 
\begin{align}
\sum_{i=1}^4b_i^2&>4\frac{a^2}{4}\\
&=a^2.\label{bi2}
\end{align}
Now, since $D$ is irreducible, we infer, by the Clebsch formula,  
\begin{equation}\label{eq:con}
\binom{a-1}{2}-\sum_{i=1}^4\frac{b_i(b_i+1)}{2}\ge 0,
\end{equation}
where $p_a(D)=\binom{a-1}{2}$ is the arithmetic genus of $D$. 
But, by \eqref{bii} and \eqref{bi2}, 
\begin{align*}
\binom{a-1}{2}-\sum_{i=1}^4\frac{b_i(b_i+1)}{2}&<\frac{a^2-3a+2}{2}-\frac{a^2}{2}-a\\
&=-\frac{5}{2}a+1<0,
\end{align*}
which contradicts \eqref{eq:con}. 

Analogously, we can deduce that there exists canonical curves $C$ contained in $S$. In fact $C$, in $S$, is 
\begin{equation*}
C\in\abs{(2k+2)H-k\sum_{i=1}^4P_i}. 
\end{equation*}
Now, it is immediate to see that the dimension of $\abs{(2k+2)H-k\sum_{i=1}^4P_i}$ is at least $5k+5$, and therefore, 
in order to prove our claim, it is sufficient to show that the general element in this linear system  
is irreducible. This fact can be obtained again by the Nakai-Moishezon Criterion: in fact, 
first of all, 
\begin{align*}
C^2&=((2k+2)H-k\sum_{i=1}^4P_i)^2\\
&= 8k+4>0. 
\end{align*}
Then, if $D\in\abs{aH-\sum_{i=1}^4b_iP_i}$ is an irreducible curve in $S$, as above, 
we deduce that 
\begin{equation*}
(2k+2)a-k\sum_{i=1}^4b_i\le 0, 
\end{equation*}
which means
\begin{align*}
\sum_{i=1}^4b_i &\ge \frac{(2k+2)a}{k}\\
&= 2a+\frac{2a}{k}\\
&> 2a, 
\end{align*}
and we can conclude as above. 

Clearly, the same kind of results (with the same proofs) 
can be obtained in the cases of the following Propositions~\ref{4gon-1} and \ref{4gon-2}. 
\end{rmk}

We then pass to case $g=3k-1$: 

\begin{prop}\label{4gon-1}
Let $C\subset\p(H^0(\omega_C)^*)$ be a tetragonal canonical curve
of genus $g=3k-1$, where $k\geq 2$ contained in a rational surface $S$ of type $[S]=2H-bF$ in a balanced scroll 
$X=\p(\OO_{\p^1}(k-2)\oplus\OO_{\p^1}(k-1)\oplus\OO_{\p^1}(k-1))$. 

Let us suppose that the projection $\pi\colon\check\p^{g-1}\dashrightarrow\check\p^2$
from the linear space generated by $2$ general points of $C$ and $k-2$ general planes of the scroll, 
restricted to $S$, is generically 1--1 (\ie $S$ is the blowing-up of this 
$\check\p^2$). 

Then,  
$\deg(S)\le \frac{4}{3}g-\frac{5}{3}$ (or $b\ge \frac{2}{3}g-\frac{13}{3}$). 

In particular, $F_{\eta_1,\eta_2}$ is a sum of 
cubes of at most  $\frac{4}{3}g-\frac{5}{3}$ 
linear forms, where 
$\eta_1,\eta_2\in H^{0}(C,\omega_{C})$ are general. 
\end{prop}

\begin{proof}
We can follow the proof of the preceding result: 
we have to consider  
a plane curve $Z$ of degree $2k+2$. 
Let $V_i$'s be the singular points of $Z$, 
and we can suppose that the points are of simple multiplicity $m_i$.

We can think of the surface $S$ containing $C$ as 
the blowing-up of $\check\p^2$ in the points $V_i$'s: $\pi_{\{V_i\}}\colon S \to \check\p^2$.
Let us denote by $H:=\pi_{\{V_i\}}^*\OO_{\check\p^2}(1)$ the hyperplane divisor of $S$ 
and by $E_i$'s the $(-1)$-curves on $S$ which correspond to the $V_i$'s. 
The complete linear system $\abs{(2k-1)H-\sum_i (1-m_i)E_i}$ gives a generically 1--1 map  
$\phi\colon\ S\to\check\p^{g-1}$ such that $C=\phi(\pi_{\{V_i\}}^{-1}(Z))$. 
Then, the adjunction formula of $C$ on $S$ in this case yields
\begin{equation*}
((2k+2)H-\sum_i m_iE_i)\cdot ((2k-1)H-\sum_i(1-m_i)E_i)=6k-4
\end{equation*}
which means, 
\begin{equation*}\label{adj-1}
\sum_i m_i(m_i-1)=2(2k^2-2k+1). 
\end{equation*}
Again, the $g^1_4$, is cut out on $Z$ by a pencil of conics, 
and as in the proof of the preceding proposition, with
the same notations, we deduce 
\begin{equation*}
\sum_{i=1}^4 n_i=4k
\end{equation*}
and 
\begin{equation*}
\sum_{i=1}^4 n^2_i\ge 4k^2,
\end{equation*}
so that 
\begin{equation*}
\sum_{i=1}^4 n_i^2-n_i\ge 4k^2-4k.
\end{equation*}
The situation here is a little more complicated, since 
$\sum_{i=1}^r m_i(m_i-1)=4k^2-4k+2> 4k^2-4k$, so we have 
two possibilities. The first one is that $r>4$, which means 
that $r=5$, $m_1=n_1=\dotsb m_4=n_4=k$ and $m_5=2$, or $r=4$, 
but then $\sum_{i=1}^4 n^2_i= 4k^2+2$ and $\sum_i m_i=4k$.

Now, 
\begin{align*}
\deg(S)&=((2k-1)H-\sum_i(1-m_i) E_i)^2\\
&= (2k-1)^2-\sum_i(m_i-1)^2\\
&=\sum_i(m_i-1)-1,
\end{align*}
but then, 
if $r=5$, we deduce $\sum_i (m_i-1)=4k-2$ and then 
we obtain that 
\begin{align*}
\deg(S)&=4k-3\\
&=\frac{4}{3}g-\frac{5}{3}; 
\end{align*}
if instead $r=4$, 
\begin{align*}
\deg(S)&=4k-5\\
&=\frac{4}{3}g-\frac{11}{3}. 
\end{align*}

\end{proof}

Finally, we analyse the case $g=3k+1$: 

\begin{prop}\label{4gon-2}
Let $C\subset\p(H^0(\omega_C)^*)$ be a tetragonal canonical curve
of genus $g=3k+1$, where $k\geq 2$, contained in a rational surface $S$ of type $[S]=2H-bF$ in a balanced scroll 
$X=\p(\OO_{\p^1}(k-1)\oplus\OO_{\p^1}(k-1)\oplus\OO_{\p^1}(k))$. 

Let us suppose that the projection $\pi\colon\check\p^{g-1}\dashrightarrow\check\p^2$
from the linear space generated by a general point of $C$ and $k-1$ general planes of the scroll, 
restricted to $S$, is generically 1--1 (\ie $S$ is the blowing-up of this 
$\check\p^2$). 

Then, 
\begin{equation*}
\deg(S)=\frac{4}{3}g-\frac{10}{3}, 
\end{equation*} 
or, equivalently, $b=\frac{2}{3}g-\frac{8}{3}$. 

In particular, $F_{\eta_1,\eta_2}$ is a sum of 
cubes of at most  $\frac{4}{3}g-\frac{10}{3}$ 
linear forms, where 
$\eta_1,\eta_2\in H^{0}(C,\omega_{C})$ are general. 
\end{prop}

\begin{proof}
As above 
we obtain 
a plane curve $Z$ of degree $2k+3$. 
Let $V_i$'s be the singular points of $Z$, 
and we can suppose that the points are of simple multiplicity $m_i$.

We can think of the surface $S$ containing $C$ as 
the blowing-up of $\check\p^2$ in the points $V_i$'s: $\pi_{\{V_i\}}\colon S \to \check\p^2$.
Let us denote by $H:=\pi_{\{V_i\}}^*\OO_{\check\p^2}(1)$ the hyperplane divisor of $S$ 
and by $E_i$'s the $(-1)$-curves on $S$ which correspond to the $V_i$'s. 
The complete linear system $\abs{2kH-\sum_i(1-m_i) E_i}$ gives a generically 1--1 map  
$\phi\colon\ S\to\check\p^{g-1}$ such that $C=\phi(\pi_{\{V_i\}}^{-1}(Z))$. 
Then, the adjunction formula of $C$ on $S$ in this case yields
\begin{equation*}
((2k+3)H-\sum_i m_iE_i)\cdot (2kH-\sum_i(1-m_i)E_i)=6k
\end{equation*}
which means, 
\begin{equation*}\label{adj-2}
\sum_i m_i(m_i-1)=4k^2. 
\end{equation*}
Again, the $g^1_4$, is cut out on $Z$ by a pencil of conics, 
and as in the proof of the preceding results, with
the same notations, we deduce 
\begin{equation*}
\sum_{i=1}^4 n_i=4k+2 
\end{equation*}
and 
\begin{equation*}
\sum_{i=1}^4 n^2_i\ge 4k^2+4k+1,
\end{equation*}
so that 
\begin{equation*}
\sum_{i=1}^4 n_i^2-n_i\ge 4k^2-1.
\end{equation*}
The situation here is that 
$\sum_{i=1}^r m_i(m_i-1)=4k^2> 4k^2-1$, and we can consider 
only the case $r=4$, for which $\sum_i (m_i-1)=4k-2$. 
Now, 
\begin{align*}
\deg(S)&=(2kH-\sum_i(1-m_i) E_i)^2\\
&= 4k^2-\sum_i(m_i-1)^2\\
&=\sum_i(m_i-1)\\
&=4k-2\\
&=\frac{4}{3}g-\frac{10}{3}. 
\end{align*}

\end{proof}



The worst estimate of the preceding proposition is $\frac{4}{3}g-\frac{5}{3}$, which is the estimate reported in the 
abstract.

\subsubsection{A generalisation}\label{sec:special}
In this subsection we will show that the bounds for the surfaces 
given in the three propositions of Subsection~\ref{ssec:stc} 
extends also to the case where $C$ is contained in a non-balanced scroll $X$. 

Instead of a check case by case, where we expect even better 
estimates, we concern ourselves only
to extend the bounds  to these special cases.
To obtain this result we consider the tetragonal loci 
${\mathcal{T}}_{g}\subset{\mathcal{M}}_{g}$ inside the moduli space of
smooth curves of genus $g$. Let 
${\mathcal{T}}^{s}_g\subset{\mathcal{T}}_{g}$ be the loci
of \emph{special} tetragonal curves, \ie not contained in a balanced scroll. It can be
shown that  ${\mathcal{T}}^{s}_g$ is a proper subscheme of 
${\mathcal{T}}_{g}$  and then, given a 
curve $C$ such that $[C]\in {\mathcal{T}}^{s}_g$, there exists an open
set $U\subset {\mathcal{M}}_{g}$, the local universal family
$\rho\colon{\mathcal{U}}\to U$ and a nontrivial morphism
$\Delta\to U$, where $\Delta\subset \CC$ is the unitary 
disk, such that the pull-back family 
$\pi\colon{\mathcal{C}}\rightarrow\Delta$ has the central fibre
$C_{0}=C$ 
which is a special
tetragonal curve and the general one is in 
${\mathcal{T}}_{g}\setminus {\mathcal{T}}^{s}_g$. More formally, we
recall that by the Hurwitz formula $\omega_{C}\simeq
f^{\star}(\omega_{ {\mathbb P}^{1} })\otimes R$, where $R$ is the
ramification divisor of the morphism $f\colon C\to \p^{1}$ which gives the $g^{1}_{4}$. Since $C$ is general in 
${\mathcal{T}}^{s}_g$, we can assume 
$2p_{\infty}^{1}+p_{\infty}^{2}+p_{\infty}^{3}=f^{\star}(\infty)$ and 
$2p_{0}^{1}+p_{0}^{2}+p_{0}^{3}=f^{\star}(0)$. In particular, we
identify a meromorphic $1$-form
$\frac{df}{f}\in
H^{0}(C,\omega(\sum_{i=1}^{3}p_{\infty}^{i}+\sum_{i=1}^{3}p_{0}^{i}))$. 
Inside the vector space $H^{1}(C, T_{C}(-\sum_{i=1}^{3}p_{\infty}^{i}-
\sum_{i=1}^{3}p_{0}^{i}))$ which parametrises the first order
deformations of $(C, p_{\infty}^{1},\dotsc ,p_{0}^{3})$, we want to
identify the subspace ${\mathcal{T}}_{f}$
of the first order deformations which extends
the $g^{1}_{4}$. So, consider $\pi\colon{\mathcal{C}}_{\epsilon}
\rightarrow{\spec}\,{\CC}\frac{[\epsilon]}{(\epsilon^{2})}$ an 
infinitesimal deformation given by $\eta\in H^{1}(C,
T_{C}(-\sum_{i=1}^{3}p_{\infty}^{i}-
\sum_{i=1}^{3}p_{0}^{i}))$. Let $P_{\infty}^{i}$, $P_{0}^{i}$ be the
infinitesimal sections which extend $p_{\infty}^{i}$, $p_{0}^{i}$,
respectively, where $i=1,2,3.$ It is a trivial remark that
the extension 
\begin{equation}\label{vedid}
0\rightarrow \OO_{C}\rightarrow
\Omega^{1}_{\mathcal{C}_{\epsilon}}({\log}(\sum_{i=1}^{3}P_{\infty}^{i}+\sum_{i=1}^{3}P_{0}^{i}))_{|C}\rightarrow
\omega_{C}(\sum_{i=1}^{3}p_{\infty}^{i}+\sum_{i=1}^{3}p_{0}^{i})\rightarrow 0
\end{equation} 
represents the class 
\begin{align*}
\eta&\in{\Ext}^{1}( \omega_{C}(\sum_{i=1}^{3}p_{\infty}^{i}+\sum_{i=1}^{3}p_{0}^{i},\OO_{C}))\\
&\simeq H^{1}(C,T_{C}(-\sum_{i=1}^{3}p_{\infty}^{i}-\sum_{i=1}^{3}p_{0}^{i})). 
\end{align*}
The fact that $f$ extends translates to
the fact that $\frac{df}{f}$ extends to a meromorphic form
$\frac{dF}{F}$ on $\mathcal{C}_{\epsilon}$. This in turns means that 
$\frac{df}{f}$ belongs to the kernel of the co-boundary operator in
the long sequence of cohomology given by the sequence \eqref{vedid}:

\begin{equation*}
    \partial_{f}\colon 
    H^{0}(C,\omega(\sum_{i=1}^{3}p_{\infty}^{i}+\sum_{i=1}^{3}p_{0}^{i}))
    \rightarrow H^{1}(C, \OO_{C}).
\end{equation*}
Then ${\mathcal{T}}_{f}$ is contained in the orthogonal subspace of
$\gen{\frac{df}{f}}$ with respect to the cup product

\begin{equation*}\label{vedidtris}
    H^{1}(C,T_{C}(-\sum_{i=1}^{3}p_{\infty}^{i}-\sum_{i=1}^{3}p_{0}^{i}))\otimes
    H^{0}(C,\omega(\sum_{i=1}^{3}p_{\infty}^{i}+\sum_{i=1}^{3}p_{0}^{i}))
    \rightarrow H^{1}(C, \OO_{C}).
\end{equation*}
Since ${\dim}_{\CC}{\mathcal{T}}_{g}=2g+3$, 
${\mathcal{T}}_{f}\subset\gen{\frac{df}{f}}^{\perp}$ and 
${\dim}_{\CC}\gen{\frac{df}{f}}^{\perp}=2g+3$ then
it follows ${\mathcal{T}}_{f}=\gen{\frac{df}{f}}^{\perp}$ and 
that ${\mathcal{T}}_{g}$ is  smooth around its general
special points. Now we turn to the pull-back family 
$\pi\colon{\mathcal{C}}\to\Delta$ of the universal family 
$\rho\colon{\mathcal{U}}\to U$. 
By flatness of $\rho\colon{\mathcal{U}}\to U$
we also have the relative canonical fibration $\pi'\colon \p\to\Delta$ and inside it a relative fibration 
$\pi'\colon{\mathcal{X}}\to\Delta$ where the general fibre is
a balanced scroll $X_{t}$.
\begin{prop}\label{specis}
Let $C\subset\p(H^0(\omega_C)^*)$ be a tetragonal canonical curve
of genus $g$. Assume that $C$ is in the closure in
$\mathcal T_g$ of the class of the curves studied in
Subsection~\ref{ssec:stc}; 
then it is contained in a surface $S$ such that: 
\begin{equation*}
\deg(S)\leq \frac{4}{3}g-\frac{5}{3}.
\end{equation*} 
\end{prop}
\begin{proof} We assume that $g=3k-1$ and we will use Proposition~\ref{4gon-1}.
The other cases are similar and give better bounds. Consider 
the relative canonical fibration $\pi'\colon \p\to\Delta$ 
and inside it a relative fibration 
$\pi'\colon{\mathcal{X}}\to\Delta$ where the general fibre is
a balanced scroll $X_{t}$ as we have done above.
Up to restrict $\Delta$ if
necessary, we can construct a fibered surface $\tau\colon\p^{1}_{\Delta}\to\Delta$ whose
fibre is a general section for the scroll $X_{t}$. By generality
of the section $\tau^{-1}(t)$ for the scroll $X_{t}$, it follows
that $\tau^{-1}(0)$ is a general section of $X_{0}$. Then, by the
generality of all the projections involved in our method, taking a 
$k-2$-multisection of $\tau\colon\p^{1}_{\Delta}\to\Delta$ which gives exactly $k-2$ points
on the fibre $\tau^{-1}(0)$, we can perform relative projections
which send the general fibre of $\pi'\colon{\mathcal{X}}
\to\Delta$ on a $\p(\OO_{\p^{1}}\oplus\OO_{\p^{1}}(1)\oplus\OO_{\p^{1}}(1))\subset\check\p^{4}$
and the
special fibre $X_{0}$ on a suitable $3$-fold.
Then, up to restrict $\Delta$ if
necessary again, we can construct a $2$-section of 
$\pi\colon{\mathcal{C}}\to\Delta$ in order to relative project to 
$\check\p^{2}$. By \cite[Proposition
III.9.8]{H} we can
construct a flat 
family of surfaces $\pi'' \colon{\mathcal S}\to
\Delta$ whose general fibre is embedded into the general fibre of 
the relative canonical fibration $\pi'\colon \p\to\Delta$. These surfaces are obtained by the blowing up of 
$\check\p^{2}\times\Delta$ along a $4$ or $5$ multisection 
by (the proof of) Proposition~\ref{4gon-1}. Since the general fibre of $\pi'' $ 
has degree $=\frac{4}{3}g-\frac{5}{3}$
then the limiting surface $S_{0}$ still satisfies this bound.
\end{proof}

\subsection{Higher order gonality}\label{sec:agg}

Now we consider an $n$-gonal curve. Following a referee's comment 
we will show that 
it is possible to construct surfaces as in Subsection~\ref{ssec:stc}, 
but it turns out that the degree of them is too big.
For simplicity, we show this in the easiest case only, 
\ie the case of the $n$-gonal curve such that $n$ 
divides the genus $g$ of $C$.

More precisely, we can write $g=(n-1)k$, and we can proceed as for Proposition~\ref{4gon}: 
so we \emph{suppose} that $C\subset\p(H^0(\omega_C)^*)$ is an $n$-gonal canonical curve
of genus $g=(n-1)k$, where $k\geq 2$, contained in a rational surface $S$ in a balanced scroll 
$X=\p(\OO_{\p^1}(k-1)\oplus\dotsb\oplus\OO_{\p^1}(k-1))$ of dimension $n-1$. 

Let us \emph{suppose also} that the projection $\pi\colon\check\p^{g-1}\dashrightarrow\check\p^2$
from the linear space generated by $(k-1)$ $\check\p^{n-2}$'s general fibres of the scroll and $(n-4)$ points of $C$ 
is generically 1--1 (\ie $S$ is the blowing-up of this 
$\check\p^2$). 

Let us find now the degree of $S$. 

Indeed, if we project from $(k-1)$ $\check\p^{n-2}$'s general fibres of the scroll,  we arrive to a $\check\p^{n-2}$
which contains the 
image of the curve, $Z'$. We have 
\begin{align*}
\deg(Z')&=2k(n-1)-2-(k-1)n\\
&=(k+1)(n-2).
\end{align*}
If then we choose $(n-4)$ points on $Z'$ and we project again from these points, 
we find a plane curve $Z$ of degree $k(n-2)+2$. 
As above, the $V_i$'s are the singular points of $Z$, of 
simple multiplicity $m_i$. 
In this case, the adjunction formula gives, in the same way as 
we obtained Formula~\eqref{adj}
\begin{equation*}\label{adjgen}
\sum_i m_i(m_i-1)=nk(nk-1)-4k^2(n-1). 
\end{equation*}

If we take a general $(n-2)$-plane $\Pi_t$, $t\in\p^1$, of the ruling of 
the scroll $\rho\colon X\to\p^1$, then $C\cap \Pi_t$ 
is given by $n$ points $\{P_{1t},\dotsc,P_{nt}\}$. Now 
take a general section of the scroll. 
We have now $(n+1)$ (general) points in  $\check\p^{n-2}\cong \Pi_t$, and 
through them there pass only one rational normal curve (of degree $(n-2)$). 
If we project these curves to $\check\p^2$, we see that 
there exists a pencil $\Lambda$ of rational curves of degree $(n-2)$ which 
cuts the $g^1_n$ on $Z$. 
Let $A_1,\dotsc,A_{(n-2)^2}$ be the base points of the
pencil; as above, we write:
\begin{equation}\label{prima}
Q'_t\mid_Z=\sum_{i=1}^{(n-2)^2} n_iA_i+ \sum_{i=1}^{n}P'_{it},
\end{equation}
where $Q'_t\in\Lambda$ and the $P'_{it}$'s are the projection
of the points $P_{it}$'s, and $t\in\p^1$, 
as in the proof of Proposition~\ref{4gon}.
Calculating the degree, we obtain, as in Formula~\eqref{equat}
\begin{equation}\label{seconda}
\sum_{i=1}^{(n-2)^2} n_i=(n-2)^2k+n-4.
\end{equation}
We note that we can write an inequality as in Formula~\eqref{ineq}.
It is not restrictive to suppose that $V_i=A_i$ 
(for the first $(n-2)^2$ $V_i$'s) and that $m_i=n_i$; then, by Formulas~\eqref{prima} 
and \eqref{seconda} we deduce
\begin{align*}
\deg(S)&=(k(n-2)-1)^2-\sum_i(m_i-1)^2\\
&=kn^2-5kn+8k-n^2+5n-7+\sum_{i>(n-2)^2}(m_i-1), 
\end{align*}
which unfortunately is greater than $2g-3$, which is the estimate of 
Iliev-Ranestad and Ciliberto-Harris, if $n\gg 0$.

\providecommand{\bysame}{\leavevmode\hbox to3em{\hrulefill}\thinspace}
\providecommand{\MR}{\relax\ifhmode\unskip\space\fi MR }
\providecommand{\MRhref}[2]{%
\href{http://www.ams.org/mathscinet-getitem?mr=#1}{#2}
}
\providecommand{\href}[2]{#2}


\begin{thebibliography}{ACGH85}

\bibitem[ACGH85]{ACGH}
Enrico Arbarello, Maurizio Cornalba, Phillip~A. Griffiths, and Joseph Harris,
\emph{Geometry of algebraic curves. {V}ol. {I}}, Grundlehren der
Mathematischen Wissenschaften [Fundamental Principles of Mathematical
Sciences], vol. 267, Springer-Verlag, New York, 1985.

\bibitem[AV03]{mv}
Marian Aprodu and Claire Voisin, \emph{{Green-Lazarsfeld's conjecture for
generic curves of large gonality}}, C. R., Math., Acad. Sci. Paris
\textbf{336} (2003), no.~4, 335--339.

 
\bibitem[BS]{bs}
Michela Brundu and Gianni Sacchiero, \emph{{Stratification of the moduli 
space of four-gonal curves}}, to appear.
 
\bibitem[CDC99]{Cas}
Gianfranco Casnati and Andrea Del Centina, \emph{On certain spaces associated to tetragonal curves of genus 7 and 8}, 
Van Oystaeyen, Freddy (ed.), Commutative algebra and algebraic geometry. Proceedings
of the Ferrara meeting in honor of Mario Fiorentini on the occasion of his retirement,
Ferrara, Italy. New York, NY: Marcel Dekker. Lect. Notes Pure Appl. Math. 206,
35--45 (1999).

\bibitem[Cha97]{Cha}
Gabriela Chaves, 
\emph{Rev\^etements ramifi\'es de la droite projective complexe}, 
Math. Z. \textbf{226} (1997), no. 1, 67--84. 


\bibitem[CH99]{ch}
Ciro Ciliberto and Joe Harris, 
\emph{Surfaces of low degree containing a general canonical curve}, 
Commun. Algebra \textbf{27}(1999) No.3, 1127--1140. 


\bibitem[DG93]{DG}
Andrea Del Centina and Alessandro Gimigliano, \emph{Scrollar invariants and resolutions of certain $d$-gonal curves}, 
Ann. Univ. Ferrara Sez. VII (N.S.) \textbf{39} (1993), 187--201.



\bibitem[Fuj90]{Fu}
Takao Fujita,
\emph{Classification theories of polarized varieties}, 
London Mathematical Society Lecture Note Series 155, 
Cambridge (UK): Cambridge University Press, 1990. 

\bibitem[Ful69]{F}
William Fulton, 
\emph{Hurwitz schemes and irreducibility of moduli of algebraic curves}, 
Ann. of Math. (2) \textbf{90} (1969), 542--575.  

\bibitem[GV06]{GV}
Sergey Gorchinskiy and Filippo Viviani, 
\emph{Families of $N$-gonal curves with maximal variation of moduli},  
 Matematiche \textbf{61} (2006) No. 1, 185--209.

\bibitem[Gre84a]{g1}
Mark~L. Green, \emph{Koszul cohomology and the geometry of projective
varieties}, J. Differential Geom. \textbf{19} (1984) no.~1, 125--171.

\bibitem[Gre84b]{g3}
\bysame, \emph{Koszul cohomology and the geometry of projective varieties.
Appendix: The nonvanishing of certain Koszul cohomology groups (by Mark Green
and Robert Lazarsfeld)}, J. Differ. Geom. \textbf{19} (1984), 125--167;
1680--171.

\bibitem[Gre84c]{g2}
\bysame, \emph{Koszul cohomology and the geometry of projective varieties.
{II}}, J. Differential Geom. \textbf{20} (1984) no.~1, 279--289.

\bibitem[Har83]{H}
Robin Hartshorne, \emph{Algebraic geometry}, {Graduate Texts in Mathematics},
vol.~52, Springer-Verlag, New York-Heidelberg-Berlin, 1983, {Corr. 3rd
printing}.

\bibitem[IK99]{ik}
Anthony Iarrobino and Vassil Kanev, \emph{Power sums, {G}orenstein algebras,
and determinant loci}, Lect. Notes Math., vol. 1721, Springer, Berlin, 1999.

\bibitem[IR01]{IR}
Atanas Iliev and Kristian Ranestad, \emph{Canonical curves and varieties of
sums of powers of cubic polynomials}, J. Algebra \textbf{246} (2001) no.~1,
385--393.

\bibitem[Mar49]{Ma}
Arturo Maroni, 
\emph{Sulle curve k-gonali}, 
Ann. Mat. Pura Appl., IV Ser. \textbf{30} (1949), 225--231.

\bibitem[Muk95]{Mu}
Shigeru Mukai, 
\emph{Curves and symmetric spaces. I}, 
Am. J. Math. \textbf{117} (1995), No. 6, 1627--1644. 

\bibitem[Sch86]{Sch}
Frank-Olaf Schreyer, \emph{Syzygies of canonical curves and special linear
series}, Math. Ann. \textbf{275} (1986), 105--137.

\bibitem[Seg28]{BSeg}
Beniamino Segre,
\emph{Sui moduli delle curve poligonali, e sopra un complemento al teorema di esistenza di Riemann}, 
Math. Ann. \textbf{100} (1928), 537--551. 


\bibitem[Voi05]{v}
Claire Voisin, \emph{Green's canonical syzygy conjecture for generic curves of
odd genus}, Compos. Math. \textbf{141} (2005), no.~5, 1163--1190.

\end{thebibliography}
\end{document}